\documentclass[reqno,12pt]{amsart}
\theoremstyle{plain}
\newtheorem{thm}{Theorem}[section]
\newtheorem{lemma}[thm]{Lemma} 
\newtheorem{prop}[thm]{Proposition}

\theoremstyle{remark}
\newtheorem{remark}[thm]{Remark}

\theoremstyle{definition}
\newtheorem{defi}[thm]{Definition}

\usepackage{amscd,amssymb,comment,epic,eepic,euscript,graphics}
\usepackage{enumerate}
\usepackage[initials]{amsrefs}
\usepackage[all]{xy}
\usepackage{color}

\newcount\theTime
\newcount\theHour
\newcount\theMinute
\newcount\theMinuteTens
\newcount\theScratch
\theTime=\number\time
\theHour=\theTime
\divide\theHour by 60
\theScratch=\theHour
\multiply\theScratch by 60
\theMinute=\theTime
\advance\theMinute by -\theScratch
\theMinuteTens=\theMinute
\divide\theMinuteTens by 10
\theScratch=\theMinuteTens
\multiply\theScratch by 10
\advance\theMinute by -\theScratch

\def\today{{\number\day\space
 \ifcase\month\or
  January\or February\or March\or April\or May\or June\or
  July\or August\or September\or October\or November\or December\fi
 \space\number\year}}

\newcommand\Ac{{\mathcal{A}}}
\newcommand\Afr{{\mathfrak A}}
\newcommand\Bc{{\mathcal{B}}}
\newcommand\Bfr{{\mathfrak B}}
\newcommand\Cpx{{\mathbf C}}
\newcommand\Dc{{\mathcal{D}}}
\newcommand\eps{\epsilon}
\newcommand\Fc{{\mathcal{F}}}
\newcommand\freeprod{\operatornamewithlimits{\ast}}
\newcommand\freeprodamalg[2]{{\displaystyle{\textstyle\;({*_{#1}})}_{#2}}}
\newcommand\gammat{{\tilde\gamma}}
\newcommand\HEu{{\EuScript H}}                   
\newcommand\Mcal{{\mathcal{M}}} 
\newcommand\Nats{{\mathbf N}}
\newcommand\oup{^{\mathrm o}}
\newcommand\phih{{\hat\phi}}
\newcommand\psih{{\hat\psi}}
\newcommand\Qc{{\mathcal{Q}}}
\newcommand\QSS{{\operatorname{QSS}}}
\newcommand\Reals{{\mathbf R}}
\newcommand\restrict{{\upharpoonright}}
\newcommand\rhohat{{\hat\rho}}
\newcommand\sigmahat{{\hat\sigma}}
\newcommand\sigmat{{\tilde\sigma}}
\newcommand\sigmathat{{\hat\sigmat}}
\newcommand\Tc{{\mathcal{T}}}
\newcommand\TQSS{{\operatorname{TQSS}}}
\newcommand\Vc{{\mathcal{V}}}
\newcommand\xh{{\hat x}}
\newcommand\yh{{\hat y}}
\newcommand\id{{\operatorname{id}}}

\begin{document}

\title[KMS quantum symmetric states]{KMS quantum symmetric states}

\author[Dykema]{Ken Dykema$^\dag$}
\address{K.\ Dykema, Department of Mathematics, Texas A\&M University, College Station, TX 77843-3368, USA}
\email{kdykema@math.tamu.edu}
\thanks{{}$^\dag$Research supported in part by NSF grant DMS-1202660.}

\author[Mukherjee]{Kunal Mukherjee}
\address{K.\ Mukherjee, Department of Mathematics, Indian Institute of Technology Mad\-ras, 
Chennai  -- 600 036, India}
\email{kunal@iitm.ac.in}


\date{Septeber 5, 2016}

\begin{abstract}
Let $A$ be a unital C$^*$-algebra and let $\sigma$ be a one-parameter automorphism group of $A$.
We consider $\QSS_\sigma(A)$, the set of all quantum symmetric states on $*_1^\infty A$ that are also KMS states
(for a fixed inverse temperature, for specificity taken to be $-1$) for the free product automorphism group $*_1^\infty\sigma$. We characterize the elements of $\QSS_\sigma(A)$, we show that $\QSS_\sigma(A)$ is a Choquet simplex whenever it is nonempty and we characterize its extreme points.
\end{abstract}

\maketitle

\section{Introduction}

For a unital C$^*$-algebra $A$, the set $\QSS(A)$ of quantum symmetric states was introduced and studied in~\cite{DKW}.
These are the states on the universal, unital free product C$^*$-algebra $*_1^\infty A$ of countably infinitely many copies of $A$,
that are invariant under the quantum permutation groups of S.\ Wang~\cite{W98}.
In~\cite{DKW}, a version of K\"ostler and Speicher's noncommutative de Finetti theorem~\cite{KS09} was proved and used to characterize
the quantum symmetric states in terms of reduced amalgamated free products of C$^*$-algebras,
and the extreme points of $\QSS(A)$ were characterized.
In~\cite{DDM14}, the set $\TQSS(A)$ of tracial quantum symmetric states was studied and was shown to be a Choquet simplex and, moreover,
the Poulsen simplex.
The extreme points of $\TQSS(A)$ were also characterized in~\cite{DDM14}.

In this paper, given a pointwise continouous one-parameter automorphism group $\sigma$ of $A$,
we study the set $\QSS_\sigma(A)$, of quantum symmetric states on $*_1^\infty A$ that are also KMS states for the free product
automorphism group $*_1^\infty\sigma$ of $*_1^\infty A$.
KMS states arose in the study of quantum statistical mechanics and are fundamental in the Tomita-Takasaki theory of
von Neumann algebras (see the introduction of~\cite{BR87} for more, and Sec.~\ref{sec:KMS} below for definitions).
Our main result (Thm.~\ref{thm:simplex}) is that $\QSS_\sigma(A)$ is a Choquet simplex; 
we also give (Thm.~\ref{thm:Psisig})
a characterization of the elements of $\QSS_\sigma(A)$ in terms of amalgamated free products and automorphism groups, and
(also in Thm.~\ref{thm:simplex}) a characterization of the extreme points of $\QSS_\sigma(A)$.

Let us mention two results that we use in the proofs of the above described results.
Thm.~\ref{thm:amalgfpKMS} shows how KMS states arise naturally using the reduced amalgmated free product construction,
and Thm.~\ref{thm:center} characterizes the center of an amalgamated free product of infinitely many
copies of a given pair consisting of a von Neumann algebra and a faithful conditional expectation.
It is a generalization of Thm. 3.3 of~\cite{DDM14}, which was in the tracial setting.

Here is a brief summary of the rest of the paper:
Sec.~\ref{sec:KMS} gives background on and recalls some results about KMS states.
Sec.~\ref{sec:KMSamalg} proves the result about KMS states of amalgamated free product C$^*$-algebras.
Sec.~\ref{sec:QSKMSS} proves the characterization of KMS quantum symmetric states.
Sec.~\ref{sec:center} describes the center of the amalgamated free product of infinitely many copies
of a von Neumann algebra.
Sec.~\ref{sec:simplex} proves that $\QSS_\sigma(A)$ is a Choquet simplex, whenever it is nonempty.

\section{KMS states}
\label{sec:KMS}

In this section, we recall the definition of KMS states and some related facts.
All of these are well known, but we provide short proofs of some of them for convenience.

Consider a C$^*$-algebra $A$ and a one-parameter, pointwise-norm continous group $(\sigma_t)_{t\in\Reals}$ of automorphisms
of $A$.
The continuity condition is that for each $x\in A$, the function $t\mapsto\sigma_t(x)$ is norm continuous.
An element $a\in A$ is said to be {\em entire analytic} if there is a function $\Cpx\ni z\mapsto \sigma_z(a)\in A$
that is differentiable with respect to the norm topology in $A$ and such that, as suggested by the notation, the restriction
to the real axis of this function is given by the automorphism group acting on $a$.
(See, for example, Prop.~2.5.21 of~\cite{BR87}.)
The facts in the following lemma are well known and are easy to verify,
using that every $A$-valued analytic function
is determined by its restriction to the real line.
\begin{lemma}\label{lem:sigmafacts}
Let $a,b\in A$, be entire analytic elements and let $z,w\in\Cpx$.
Then
\begin{enumerate}[(i)]
\item $a+b$ is an entire analytic element  and $\sigma_z(a+b)=\sigma_z(a)+\sigma_z(b)$;
\item $ab$ is an entire analytic element and $\sigma_z(ab)=\sigma_z(a)\sigma_z(b)$;
\item $a^*$ is an entire analytic element and $\sigma_z(a^*)=\sigma_{\overline z}(a)^*$;
\item $\sigma_w(a)$ is an entire analytic element and $\sigma_z\big(\sigma_w(a)\big)=\sigma_{z+w}(a)$.
\end{enumerate}
\end{lemma}

Let $\Afr_\sigma$ be the set of all entire analytic elements of $A$.
From the above, it follows that $\Afr_\sigma$ is a $*$-subalgebra of $A$.
By Prop.~2.5.22 of~\cite{BR87}, this subalgebra is dense in $A$.

\begin{lemma}\label{lem:Eanalytic}
Let $B$ be a C$^*$-subalgebra of a C$^*$-algebra $A$ and suppose $E:A\to B$ is a conditional expectation.
Suppose $\sigma=(\sigma_t)_{t\in\Reals}$ is a one-parameter automorphism group of $A$ such that $E\circ\sigma_t=\sigma_t\circ E$
for all $t\in\Reals$.
Suppose $a\in A$ is entire analytic.
Then $E(a)$ is entire analytic, and $\sigma_z(E(a))=E(\sigma_z(a))$ for all $z\in\Cpx$.
\end{lemma}
\begin{proof}
Since $E$ has norm $1$, the map $z\mapsto E(\sigma_z(a))$ is $B$-valued holomorphic.
When $z=t\in\Reals$, then we have
$E(\sigma_t(a))=\sigma_t(E(a))$.
\end{proof}

A proof of the next result is contained in the proof of Prop. 5.3.7 of~\cite{BR97}.

\begin{prop}\label{prop:KMS}
Let $\beta\in\Reals\backslash\{0\}$ and let $\phi$ be a state of $A$.
Then the following are equivalent:
\begin{enumerate}[(a)]
\item There is a $\sigma$-invariant dense $*$-subalgebra $\Bfr$ of $\Afr_\sigma$, such that for all $a\in A$ and all $b\in\Bfr$,
$\phi(a\sigma_{\beta i}(b))=\phi(ba)$.
\item For all $a\in A$ and all $b\in\Afr_{\sigma}$,
$\phi(a\sigma_{\beta i}(b))=\phi(ba)$.
\item\label{it:KMSF}
If $\beta>0$, let 
\[
D_\beta=\{z\in\Cpx\mid 0<\Im z<\beta\}
\]
and if $\beta<0$, let
\[
D_\beta=\{z\in\Cpx\mid\beta<\Im z<0\}.
\]
Let $\overline{D_\beta}$ denote the closure of $D_\beta$.
Then for all $a,b\in A$, there is a bounded and continuous function $f_{a,b}:\overline{D_\beta}\to\Cpx$ that is holomorphic
on $D_\beta$ and such that for all $t\in\Reals$,
\[
f_{a,b}(t)=\phi\big(a\sigma_t(b)\big),\qquad f_{a,b}(t+i\beta)=\phi\big(\sigma_t(b)a\big).
\]
\end{enumerate}
\end{prop}

\begin{defi}\label{def:KMS}
A state $\phi$ of $A$ is said to 
satisfy the {\em Kubo-Martin-Schwinger} or {\em KMS condition}
with respect to $\sigma$ at value $\beta\in\Reals$
if it satisfies the above equivalent conditions.
\end{defi}
Note that, if a state satisfies the KMS-condition at value $\beta\ne0$ for the automorphism group $\sigma$, 
then it satisfies the KMS-condition at value $\beta=-1$ for an aumorphism group obtained from $\sigma$ by reparameterization.
Thus, we confine ourselves to the case $\beta=-1$, and we call such states
$\sigma$-KMS states, or, if $\sigma$ is clear
from context, simply {\em KMS states}.
(See Def.~5.3.1 of~\cite{BR97}.)

\begin{remark}\label{rem:invariant}
By Prop.~5.3.3 of \cite{BR97}, every $\sigma$-KMS state $\phi$ is $\sigma$-invariant, namely, $\phi(\sigma_t(a))=\phi(a)$
for every $a\in A$ and $t\in\Reals$.
It then follows, when $a$ is an entire analytic element, that $\phi(\sigma_z(a))=\phi(a)$ for all $z\in\Cpx$.
\end{remark}

\begin{remark}\label{rem:homomKMS}
From condition~\eqref{it:KMSF} of Prop.~\ref{prop:KMS}, we see that if $B$ is a unital C$^*$-algebra with one-parameter
automorphism group $\gamma=(\gamma_t)_{t\in\Reals}$ and if $\theta:B\to A$ is a $*$-homomorphism such that
$\sigma_t\circ\theta=\theta\circ\gamma_t$ for all $t\in\Reals$, then $\phi\circ\theta$ is a $\gamma$--KMS state of $B$.
\end{remark}

The following easy result will be useful in proving that certain states are KMS states.
\begin{lemma}\label{lem:subalg}
Let $\sigma$ be a one-parameter automorphism group of a
C$^*$-algebra $A$. 
Let 
\[
\Bfr=\{b\in\Afr_{\sigma}\mid\forall a\in A,\,\phi(a\sigma_{-i}(b))=\phi(ba)\}.
\]
Then $\Bfr$ is a subalgebra of $\Afr_{\sigma}$.
\end{lemma}
\begin{proof}
Clearly, $\Bfr$ is a vector subspace of $\Afr_{\sigma}$. 
We need only show that $\Bfr$ is closed under multiplication.
Let $b_1,b_2\in\Bfr$ and take $a\in A$.
Then we have
\[
\phi\big(a\sigma_{-i}(b_1b_2)\big)=\phi\big(a\sigma_{-i}(b_1)\sigma_{-i}(b_2)\big)
=\phi\big(b_2a\sigma_{-i}(b_1)\big)=\phi(b_1b_2a).
\]
Thus, $b_1b_2\in\Bfr$.
\end{proof}

\smallskip
The following result is well known (as are all in this section).
For convenience, we provide a proof.
\begin{lemma}\label{lem:faithful}
Let $\phi$ be a $\sigma$-KMS state of a C$^*$-algebra $A$ and suppose that the Gelfand-Naimark-Segal $($GNS$)$ representation $\pi_\phi$
of $\phi$ is faithful on $A$.
Then $\phi$ is faithful on $A$.
Moreover, the unique normal extension $\phih$ of $\phi$ to the von Neumann algebra $\Mcal=\pi_\phi(A)''$
generated by the image of the GNS representation of $A$, is faithful on $\Mcal$.
\end{lemma}
\begin{proof}
Using the Cauchy-Schwarz inequality, we easily see that faithfulness of $\phi$ is equivalent to the existence, for every
nonzero $a\in A$,
of an element $b\in A$ such that $\phi(ba)\ne0$.

Let $a\in A$, $a\ne0$.  Since the GNS representation of $\phi$ is faithful,
there exist $x,y\in A$ such that $\phi(y^*ax)\ne0$.
Thus, $\phi(y^*a\sigma_t(x))\ne0$ for all real values $t$ in some neighborhood of $0$.
Let $f=f_{y^*a,x}$ be the bounded, continuous function on $\overline{D_{-1}}$ as described in~\eqref{it:KMSF} of Prop.~\ref{prop:KMS}.
By the Schwarz reflection principle and uniqueness of analytic continuation, it follows that
we cannot have that $f(t-i)=0$ for all $t\in\Reals$ (for more details, see, e.g., Prop.~5.3.6 of~\cite{BR97}).
Thus, there exists $t\in\Reals$ such that $0\ne f(t-i)=\phi(\sigma_t(x)y^*a)$.
Taking $b=\sigma_t(x)y^*$, this completes the proof that $\phi$ is faithful.

We now show faithfulness of $\phih$.
Note that the automorphism group $\sigma$ of $A$ extends to a unitarily implimented, pointwise strong-operator-topology
continuous, one-parameter automorphism group $\sigmahat=(\sigmahat_t)_{t\in\Reals}$ of $\Mcal$.
Thus, the state, $\phih$ is $\sigmahat$-KMS, according to Defn. 5.3.1 of~\cite{BR97}.
Now, using the part of Prop~5.3.7 of~\cite{BR97} that applies to  W$^*$-algebras, the proof proceeds as before.
In particular, for any nonzero $a\in\Mcal$, there exist $x,y\in\Mcal$ such that $\phih(y^*ax)\ne0$.
Then there is a function $f$
that is continuous on $\overline{D_{-1}}$ and analytic in $D_{-1}$ 
so that
\[
\forall t\in\Reals\qquad f(t)=\phih(y^*a\sigmahat_t(x)),\qquad f(t-i)=\phih(\sigmahat_t(x)y^*a).
\]
As before, this ensures that $\phih(\sigmahat_t(x)y^*a)\ne0$ for some $t\in\Reals$, and this implies $\phih(a^*a)\ne0$.
\end{proof}

\section{KMS states and amalgamated free products}
\label{sec:KMSamalg}

The next result is that free products of KMS states are KMS states.
\begin{thm}\label{thm:fpKMS}
Let $I$ be a nonempty set and for all $\iota\in I$ let $A_\iota$ be a unital C$^*$-algebra with one-parameter automorphism group
$\sigma^{(\iota)}$ and let $\phi_\iota$ be
a $\sigma^{(\iota)}$-KMS state on $A_\iota$.
Let
\[
(A,\phi)=\freeprod_{\iota\in I}(A_\iota,\phi_\iota)
\]
be the reduced free product of C$^*$-algebras.
Since each automorphism group $\sigma^{(\iota)}$ leaves $\phi_\iota$ invariant, we have the free product automorphism group
$\sigma$ on $A$ given by
$\sigma_t=\freeprod_{\iota\in I}\sigma^{(\iota)}_t$ for all $t\in\Reals$.
Then $\phi$ is a $\sigma$-KMS state on $A$.
\end{thm}

The proof of the above result is straightforward.
The result is also a special case of the following one.
We will give a proof of the following result
directly, using C$^*$-algebra concepts.
However, it could also be proved by using Ueda's description~\cite{U99}
of the modular automorphism group in an amalgamated free product
and the fact, (see \cite{T03}, Thm. VIII.1.2) that, for a given faithful normal state on a von Neumann algebra,
the state is KMS for the modular automorphism group of the state and for no other one-parameter automorphism group.

\begin{thm}\label{thm:amalgfpKMS}
Let $B$ be a unital C$^*$-algebra.
Let $I$ be a nonempty set and for all $\iota\in I$ let $A_\iota$ be a unital C$^*$-algebra containing $B$
as a unital C$^*$-subalgebra
and let $E_\iota:A_\iota\to B$ be a conditional expectation having faithful GNS representation.
Let
\[
(A,E)=\freeprodamalg{B}{\iota\in I}(A_\iota,E_\iota)
\]
be the reduced amalgamated free product of C$^*$-algebras.

Suppose, for each $\iota\in I$,
$\sigma^{(\iota)}$ is a one-parameter automorphism group of $A_\iota$ and suppose
$\sigma^{(\iota)}_t\circ E_\iota=E_\iota\circ\sigma^{(\iota)}_t$
for all $t\in\Reals$.
Suppose also, for all $t\in\Reals$, that the automorphism of $B$ obtained by restricting $\sigma^{(\iota)}_t$ to $B$
is the same for all $\iota\in I$.
There is a one-parameter automorphism group $\sigma$ given by $\sigma_t=\freeprod_{\iota\in I}\sigma^{(\iota)}_t$ for all $t\in\Reals$,
which we call the free product automorphism group of the $\sigma^{(\iota)}$ for $\iota\in I$. 

Suppose $\rho$ is a state on $B$ such that, for every $\iota\in I$, $\rho\circ E_\iota$ is a $\sigma^{(\iota)}$-KMS state of $A_\iota$.
Then $\rho\circ E$ is a $\sigma$-KMS state of $A$.
\end{thm}
\begin{proof}
The existence of the free product automorphism $\sigma_t$ for each $t$ is standard and is not difficult to verify.
These clearly form an automorphism group.
The required pointwise norm-continuity of this group is easy to verify by considering actions on the dense subalgebra of $A$ generated by
$\bigcup_{i\in I}A_\iota$.

Let $\phi=\rho\circ E$.
Let $\Afr^{(\iota)}$ denote the $*$-algebra of elements in $A_\iota$ that are analytic for the auotomorphism group $\sigma^{(\iota)}$
and let $\Afr$ denote the $*$-algebra of elements of $A$ that are analytic for the automorphism group $\sigma$.
Let
\[
\Bfr=\{b\in\Afr\mid \forall a\in A, \phi(a\sigma_{-i}(b))=\phi(ba)\}.
\]
By Prop.~\ref{prop:KMS}, it will suffice to show that $\Bfr$ is dense in $A$.
By Lemma~\ref{lem:subalg}, it will suffice to show that $\Afr^{(\iota)}\subseteq\Bfr$ for all $\iota\in I$.
Indeed, each $\Afr^{(\iota)}$ is dense in $A_\iota$, so the algebra they generate will be dense in $A$.

Fix $\iota\in I$ and $b\in\Afr^{(\iota)}$.
Note that $\sigma_{-i}(b)\in A_\iota$ (see Lemma~\ref{lem:Eanalytic}).
In order to show that $b\in\Bfr$, it will suffice to show that the equality
\begin{equation}\label{eq:phiab}
\phi(a\sigma_{-i}(b))=\phi(ba)
\end{equation}
holds
(i) for all $a\in A_\iota$ and (ii) for all $a$ of the form $a=a_1a_2\cdots a_n$ for $n\ge1$ and $a_j\in A_{\iota(j)}\cap\ker E_{i(j)}$,
for some $\iota(1),\ldots,\iota(n)\in I$ with $\iota(j)\ne\iota(j+1)$ for all $1\le j\le n-1$,
and with either $n\ge2$ or with $\iota(1)\ne\iota$.
Indeed, the set of all such elements $a$ densely spans the algebra $A$.
If $a\in A_\iota$ as in~(i) above, then the identity~\eqref{eq:phiab} holds because 
the restriction of $\phi$ to $A_\iota$ equals $\rho\circ E_\iota$, which is KMS by hypothesis.
If $a$ is as in~(ii) above, then by freeness, we see that $E(ba)$ and $E(a\sigma_{-i}(b))$ are both zero,
so both sides of~\eqref{eq:phiab} vanish.
\end{proof}

\section{KMS quantum symmetric states}
\label{sec:QSKMSS}

Let $A$ be a unital C$^*$-algebra and let
 $*_1^\infty A$ denote the universal unital free product C$^*$-algebra of $A$ with itself infinitely many times.
Recall that $\QSS(A)$ denotes the set of states of $*_1^\infty A$ that are quantum symmetric, and
Thm.~7.3 of~\cite{DKW} gives a characterization of these in terms of amalgamated free products.
In particular, it provides a bijection 
\begin{equation}\label{eq:Psi}
\Psi:\Vc(A)\to\QSS(A),
\end{equation}
where $\Vc(A)$ is the set of all quintuples $(B,D,F,\theta,\rho)$ with $D$ a unital C$^*$-subalgebra of $B$,
$F:B\to D$ a conditional expectation, $\theta:A\to B$ a unital $*$-homomorphism and $\rho$ a state on $D$, subject to some other
conditions (see Defn.~7.1 of~\cite{DKW} for the full description);
the state $\Psi((B,D,F,\theta,\rho))$ arises as the composition $\rho\circ G\circ(*_1^\infty\theta)$,
where $(Y,G)=(*_D)_1^\infty(B,F)$ is the reduced amalgamated free product of C$^*$-algebras.

\begin{defi}
Let $A$ be a unital C$^*$-algebra and
suppose $\sigma=(\sigma_t)_{t\in\Reals}$ is a continuous one-parameter automorphism group of $A$.
Let $*_1^\infty\sigma$ denote the one-parameter automorphism group of $*_1^\infty A$ given by
$(*_1^\infty\sigma)_t=*_1^\infty\sigma_t$.
Let $\QSS_\sigma(A)$ denote the set of all $\psi\in\QSS(A)$ that are $(*_1^\infty\sigma)$-KMS states.
To be short, we will call these the {\em KMS quantum symmetric states} of $(A,\sigma)$.
\end{defi}

In this section, we describe the restriction of the bijection~\eqref{eq:Psi} that characterizes the KMS quantum symmetric states.
The KMS condition imposes further restrictions on elements of $\Vc(A)$ that result in simplifications as compared to
Defn.~7.1 of~\cite{DKW}.
This makes it appropriate to consider the following new definition.
\begin{defi}\label{def:VsigmaA}
Let $A$ be a unital C$^*$-algebra and let $\sigma=(\sigma_t)_{t\in\Reals}$ be a one-parameter group of automorphisms of $A$.
Let $\Vc_\sigma(A)$ be the set of all equivalence classes of sextuples $(B,D,F,\theta,\gamma,\rho)$, such that
\begin{enumerate}[(i)]
\item\label{it:B} $B$ is a unital $C^*$-algebra,
\item $D$ is a unital $C^*$-subalgebra of $B$,
\item $F:B\to D$ is a faithful conditional expectation onto $D$,
\item $\theta:A\to B$ is a unital $*$-homomorphism,
\item\label{it:ADgen} $\theta(A)\cup D$ generates $B$ as a $C^*$-algebra,
\item\label{it:Dsmallest} $D$ is the smallest unital $C^*$-subalgebra of $B$ that satisfies
\begin{equation}\label{eq:Fdas}
F\big(x_0\theta(a_1)x_1\cdots\theta(a_n)x_n\big)\in D,
\end{equation}
whenever $n\in\Nats$, $x_0,\ldots,x_n\in D$ and $a_1,\ldots,a_n\in A$,
\item\label{it:gammagp} $\gamma=(\gamma_t)_{t\in\Reals}$ is a one-parameter automorphism group of $B$,
\item\label{it:gammaF} for all $t\in\Reals$, $\gamma_t\circ\theta=\theta\circ\sigma_t$ and $\gamma_t\circ F=F\circ\gamma_t$,
\item\label{it:gammaKMS} $\rho$ is a faithful state on $D$ and $\rho\circ F$ is a $\gamma$-KMS state of $B$,
\end{enumerate}
and where sextuples $(B,D,F,\theta,\gamma,\rho)$ and $(B',D',F',\theta',\gamma',\rho')$ are defined to be equivalent
if and only if there is a $*$-isomorphism $\pi:B\to B'$ sending $D$ onto $D'$ and so that $\pi\circ F=F'\circ\pi$,
$\pi\circ\theta=\theta'$, $\rho'\circ\pi\restrict_D=\rho$ and $\gamma_t'\circ\pi=\pi\circ\gamma_t$ for all $t\in\Reals$.
\end{defi}

\begin{remark}
\begin{enumerate}[(a)]
\item Just as explained in Rem.~7.2(a) of~\cite{DKW}, for any given $A$
it is possible to select a Hilbert space
$\HEu$ having a basis of sufficiently large cardinality so that every sextuple $(B,D,F,\theta,\gamma,\rho)$ in Defn.~\ref{def:VsigmaA}
is equivalent to $(B',D',F',\theta',\gamma',\rho')$, where $B'$ is a C$^*$-algebra in $\Bc(\HEu)$.
In order to avoid set theoretic difficulties, we should restrict to sextuples involving C$^*$-subalgebras of $\Bc(\HEu)$,
but we will use sloppy language and refer to ``all sextuples.''
\item  Just as explained in Rem.~7.2(b) of~\cite{DKW}, we will suppress the notation for equivalence classes and just write 
$(B,D,F,\theta,\gamma,\rho)\in\Vc_\sigma(A)$.
\end{enumerate}
\end{remark}

\begin{lemma}\label{lem:autounique}
Suppose $(B,D,F,\theta)$ satisfy properties \eqref{it:B}--\eqref{it:Dsmallest} of Defn.~\ref{def:VsigmaA}
and $\alpha$ is an automorphism of $A$.
Then there is at most one automorphism $\gamma$ of $B$ such that 
$\gamma\circ\theta=\theta\circ\alpha$ and $\gamma\circ F=F\circ\gamma$.
\end{lemma}
\begin{proof}
Using property~\eqref{it:Dsmallest}, we see $D=\overline{\bigcup_{n=1}^\infty D_n}$, where $D_0=\Cpx1$ and where, for $n\ge1$,
$D_n$ is defined recursively to be the algebra generated by
\[
\big\{F\big(\theta(a_1)d_1\cdots\theta(a_{k-1})d_{k-1}\theta(a_k)\big)\mid k\in\Nats,\,a_1,\ldots,a_k\in A,\,d_1,\ldots,d_{k-1}\in D_{n-1}\big\}.
\]
We see by induction on $n$ that $\gamma$ is uniquely determined on $D_n$.
Indeed, the case $n=0$ is clear, while the induction step follows from the identity
\begin{multline*}
\gamma\big(F\big(\theta(a_1)d_1\cdots\theta(a_{k-1})d_{k-1}\theta(a_k)\big)\big) \\
=F\big(\theta(\alpha(a_1))\gamma(d_1)\cdots\theta(\alpha(a_{k-1}))\gamma(d_{k-1})\theta(\alpha(a_k))\big).
\end{multline*}
This implies that $\gamma$ is uniquely determined on $D$.
It is clearly uniquely determined on $\theta(A)$ and, therefore also on $C^*(\theta(A)\cup D)=B$.
\end{proof}

\begin{thm}\label{thm:Psisig}
There is a bijection
\[
\Psi_\sigma:\Vc_\sigma(A)\to\QSS_\sigma(A)
\]
defined by
\begin{equation}\label{eq:Psisig}
\Psi_\sigma(B,D,F,\theta,\gamma,\rho)=\Psi(B,D,F,\theta,\rho),
\end{equation}
where $\Psi$ is the bijection in~\eqref{eq:Psi} from~\cite{DKW}.
\end{thm}
\begin{proof}
From Lemma~\ref{lem:autounique}, it is clear that the map $\Vc_\sigma(A)\ni(B,D,F,\theta,\gamma,\rho)\mapsto(B,D,F,\theta,\rho)$
that forgets $\gamma$ is injective. Moreover, the faithfulness of $F$ and $\rho$, assumed in Defn.~\ref{def:VsigmaA}, guarantees that $(B,D,F,\theta,\rho)\in\Vc(A)$ (see \eqref{eq:Psi}).
Thus, the definition~\eqref{eq:Psisig} of $\Psi_\sigma$ makes sense and defines an injection from $\Vc_\sigma(A)$ into $\QSS(A)$.

To show that $\Psi_\sigma$ maps $\Vc_\sigma(A)$ into $\QSS_\sigma(A)$, suppose $\psi=\Psi_\sigma(B,D,F,\theta,\gamma,\rho)$
and let us show that $\psi$ is a $(*_1^\infty\sigma)$-KMS state.
By Thm.~7.3 of~\cite{DKW}, the state $\psi=\Psi(B,D,F,\theta,\rho)$ is constructed as $\psi=\rho\circ G\circ(*_1^\infty\theta)$,
where $(Y,G)=(*_D)_1^\infty(B,F)$ is the amalgamated free product of C$^*$-algebras.
Considering the assumptions \eqref{it:gammagp}--\eqref{it:gammaKMS} of Defn.~\ref{def:VsigmaA} and
applying Thm.~\ref{thm:amalgfpKMS}, we see that,
with respect to the one-parameter automorphism group $\gammat:=(*_1^\infty\gamma_t)_{t\in\Reals}$,
the state $\rho\circ G$ of the C$^*$-algebra $Y$ is $\gammat$--KMS.
Now noting that $(*_1^\infty\gamma_t)\circ(*_1^\infty\theta)=(*_1^\infty\theta)\circ(*_1^\infty\sigma_t)$ and
applying the observation made in Rem.~\ref{rem:homomKMS}, we conclude that
$\psi$ is a $(*_1^\infty\sigma)$-KMS state.

In order to show that $\Psi_\sigma$ maps $\Vc_\sigma(A)$ onto $\QSS_\sigma(A)$, let $\psi\in\QSS_\sigma(A)$ and, by virtue of the bijection~\eqref{eq:Psi}, let
$(B,D,F,\theta,\rho)\in\Vc(A)$ be such that $\psi=\Psi(B,D,F,\theta,\rho)$.
We will construct a one-parameter automorphism group $\gamma=(\gamma_t)_{t\in\Reals}$ of $B$ so that
$(B,D,F,\theta,\gamma,\rho)\in\Vc_\sigma(A)$, which will complete the proof.
For this, we will work with the construction of $(B,D,F,\theta,\rho)$ from $\psi$, found in Subsection~5.1 of~\cite{DKW}.
Write $\sigmat_t$ for the automorphism $*_1^\infty\sigma_t$ of $*_1^\infty A$ and $\sigmat$ for the one parameter automorphism group $(\sigmat_t)_{t\in\Reals}$.
Since $\psi$ is $\sigmat$-KMS, we have $\psi\circ\sigmat_t=\psi$ for all $t$.
Recall that, for $i\in\Nats$,
$\lambda_i:A\to*_1^\infty A$ denotes the $*$-homomorphism sending $A$ onto the $i$-th copy in the free product construction,
that $\Mcal_\psi$ denotes the von Neumann algebra generated by the image of $*_1^\infty A$ under the GNS representation $\pi_\psi$
of $\psi$ and that $\pi_i=\pi_\psi\circ\lambda_i$.
Since $\sigmat_t$ leaves $\psi$ invariant and is point-norm continuous in $t$,
it yields a unitarily implemented $*$-automorphism $\sigmathat_t$ of $\Mcal_\psi$
satisfying  $\sigmathat_t\circ\pi_\psi=\pi_\psi\circ\sigma_t$ and, thus,
\begin{equation}\label{eq:sigpi}
\sigmathat_t\circ\pi_i=\pi_i\circ\sigma_t, \text{ }t\in\Reals,
\end{equation}
for every $i\in\Nats$.
The implementing unitary is $U_t$ defined by $\xh\mapsto(\sigma_t(x))\hat{\;}$ for $x\in*_1^\infty A$, and is strongly continuous in $t$ (and this choice is unique on fixing the image of the standard vacuum vector).

Recall that the tail algebra is
\[
\Tc_\psi =\bigcap_{N=1}^\infty W^*\big(\bigcup_{i=N}^\infty\pi_i(A)\big).
\]
It is, clearly, globally invariant under $\sigmathat_t$.
Thus, $\Qc_0$, which is the $*$-algebra generated by $\Tc_\psi\cup\bigcup_{i=1}^\infty\pi_i(A)$, is also globally invariant under $\sigmathat_t$
and, therefore, so is its closure $\Qc_\psi$.
Lemma~5.1.9 of~\cite{DKW} constructed a $*$-automorphism $\alpha$ of $\Qc_\psi$ that restricts to the identity on $\Tc_\psi$
and satisfies, for every $i\in\Nats$, $\alpha\circ\pi_i=\pi_{i+1}$,
and Lemma~5.1.10 of~\cite{DKW} constructs a conditional expectation $F_\psi$ from $\Qc_\psi$ onto $\Tc_\psi$ by setting
\[
F_\psi(x)=\operatorname{WOT--}\lim_{n\to\infty}\alpha^n(x),
\]
where ``WOT'' means weak operator topology.
We have
\[
\alpha\circ\sigmathat_t(\pi_i(a))=\alpha\circ\pi_i\circ\sigma_t(a)=\pi_{i+1}\circ\sigma_t(a)=\sigmathat_t\circ\pi_{i+1}(a)
=\sigmathat_t\circ\alpha(\pi_i(a)), \text{ }a\in A.
\]
This implies that $\alpha$ commutes with $\sigmathat_t$.
Since $\sigmathat_t$ is WOT to WOT continuous, this yields $F_\psi\circ\sigmathat_t=\sigmathat_t\circ F_\psi$.
By definition (see Observation~5.3.3 of~\cite{DKW}) the tail C$^*$-algebra is
\[
D=\overline{\bigcup_{n=0}^\infty D_{\psi,n}}\subseteq\Tc_\psi,
\]
where $D_{\psi,0}=\Cpx1$ and where
$D_{\psi,n}$ is defined recursively to be the algebra generated by
\begin{multline*}
\big\{F_\psi\big(\pi_{i_1}(a_1)d_1\cdots\pi_{i_{k-1}}(a_{k-1})d_{k-1}\pi_{i_k}(a_k)\big)\,\big| \\
k\in\Nats,\,i_1,\ldots,i_k\in\Nats,\,a_1,\ldots,a_k\in A,\,
d_1,\ldots,d_{k-1}\in D_{\psi,n-1}\big\}.
\end{multline*}
Now using~\eqref{eq:sigpi}, we easily see by induction on $n$ that
\begin{itemize}
\item[$\bullet$] $\sigmathat_t(D_{\psi,n})\subseteq D_{\psi,n}$ for every $t\in\Reals$,
\item[$\bullet$] $\sigmathat_s\circ\sigmathat_t(d)=\sigmathat_{s+t}(d)$, for every $s,t\in\Reals$ and $d\in D_{\psi,n}$,
\item[$\bullet$] the maping $t\mapsto\sigmathat_t(d)$ is norm continuous, for every $d\in D_{\psi,n}$.
\end{itemize}
Thus, the restriction of $\sigmathat_t$ to $D$ yields a continuous one-parameter automorphism group of $D$.
The algebra $B$ is by definition the C$^*$-algebra generated by $D\cup\pi_1(A)$, the $*$-homomorphism $\theta:A\to B$
is just $\pi_1$ and the conditional expectation $F:B\to D$ is just the restriction of $F_\psi$ to $B$.
We let $\gamma_t$ be the restriction of $\sigmathat_t$ to $B$ and we have that $\gamma=(\gamma_t)_{t\in\Reals}$
is a continuous, one-parameter automorphism group of $B$ satisfying $\gamma_t\circ\theta=\theta\circ\sigma_t$
and $\gamma_t\circ F=F\circ\gamma_t$.

We will now show that $\rho\circ F$ is a $\gamma$-KMS state of $B$.
Even more, letting $(Y,G)=(*_D)_1^\infty(B,F)$ be the amalgamated free product of infinitely many copies of $(B,F)$
and letting $\gammat$ be the one-parameter automorphism group of $Y$ given by $\gammat_t=*_1^\infty\gamma_t$,
we will show that $\rho\circ G$ is a $\gammat$-KMS state of $Y$.
As in~\cite{DKW}, we let $\psih$ denote the vector state on $\Mcal_\psi$ for the vector coming from the identity element of $*_1^\infty A$,
so that $\psih\circ\pi_\psi=\psi$.
Since $\psi$ is a $\sigmat$-KMS state on $*_1^\infty A$, it follows that $\psih$ is a $\sigmathat$-KMS state and, in particular (see Prop. 5.3.7 of~\cite{BR97}, {\em cf}.\/ Prop.~\ref{prop:KMS}\eqref{it:KMSF}), for every $x,y\in\Mcal_\psi$, there is a
continuous function $f_{x,y}:\overline{D_{-1}}\to\Cpx$ that is holomorphic on $D_{-1}$ and such that for all $t\in\Reals$,
\[
f_{x,y}(t)=\psih\big(x\sigmathat_t(y)\big),\qquad f_{x,y}(t-i)=\psih\big(\sigmathat_t(y)x\big).
\]
Thus, the restriction of $\psih$ to any $\sigmathat$-invariant C$^*$-subalgebra of $\Mcal_\psi$ is KMS.
By Thm.~7.3 of~\cite{DKW}, the C$^*$-algebra $Y$ is identified with the subalgebra $\pi_\psi(*_1^\infty A)$ of $\Mcal_\psi$,
where $\rho\circ G$ is identified with the restriction of $\psih$ and where $\gammat$ is identified with the restriction of $\sigmathat$.
It follows that $\rho\circ G$ is $\gammat$-KMS.

By Rem. 7.2(c) of~\cite{DKW}, the GNS representation of $\rho\circ G$ is faithful.
Therefore, by Lemma~\ref{lem:faithful}, $\rho\circ G$ is faithful.
Thus, $\rho\circ F$ is faithful, so both $\rho$ and $F$ are faithful.
Finally, since $\rho\circ G$ is $\gammat$-KMS, we get that $\rho\circ F$ is $\gamma$-KMS.
It follows that all conditions of Defn.~\ref{def:VsigmaA} are satisfied and $(B,D,F,\theta,\gamma,\rho)\in\Vc_\sigma(A)$.
\end{proof}

\section{Centers of certain amalgamated free product von Neumann algebras}
\label{sec:center}

In this section, we return to a topic visited in Sec.~3 of~\cite{DDM14} and we prove results in greater generality.
This problem is to describe the center of a von Neumann algebra realized as the amalgamated free product over
a von Neumann algebra $\Dc$ of infinitely many
copies of $(\Bc,E)$, where $\Bc$ is a von Neumann algebra and $E$ is a normal, faithful conditional expectation onto a unital copy $\Dc\subseteq\Bc$.
The answer, unsurprisingly, is that the center is precisely the intersection of the center of $\Bc$ with $\Dc$;
in~\cite{DDM14}, we proved this under the hypothesis that there exists a tracial state
$\tau_D$ such that $\tau_D\circ E$ is a faithful trace on $\Bc$,
but here we will only assume that $E$ is faithful and that $\Dc$ has a normal, faithful state.

Throughout this section, we let
$\Dc\subseteq\Bc$ be a unital inclusion of
von Neumann algebras with a normal, faithful,
conditional expectation $E:\Bc\to\Dc$.

For an element $x$ of a von Neumann algebra, we will let $[x]$ denote the range projection of $x$.
Thus, $[x]$ is the orthogonal projection onto the closure of the range of $x$, considered as a Hilbert space operator,
and it belongs to the von Neumann algebra generated by $x$.
The notation $Z(\Ac)$ means the center of $\Ac$.

The following is a generalization of Lemma~3.1 of~\cite{DDM14}.
In fact, though that lemma formally assumed the existence of a trace as described above, it only depended
on the faithfulness of the conditional expectation $E$.
Thus, the main (and, actually, quite mild) innovation in the following version is the use of a dense subalgebra $A$, which
will be used in the proof of Thm.~\ref{thm:center} below.

\begin{lemma}\label{lem:suppProj}
Let $E:\Bc\to\Dc$ be a faithful, normal, conditional expectation as described above.
Suppose $A$ is a strong-operator-topology dense $*$-subalgebra of $\Bc$ such that $E(A)\subseteq A$.
Let
\[
q_A=\bigvee\;\{[E(b^*b)]\mid b\in A\cap\ker E\}.
\]
Then $q_A$ is independent of the choice of $A$, and, writing $q$ for this projection, we have
$q\in\Dc\cap Z(\Bc)$ and $(1-q)\Bc=(1-q)\Dc$.
\end{lemma}
\begin{proof}
We only need to show
\begin{equation}\label{eq:qE}
q_A=\bigvee\;\{[E(b^*b)]\mid b\in\ker E\},
\end{equation}
because then the proof of Lemma~3.1 of~\cite{DDM14} shows that $q$ has the desired properties.

Let $q(E)$ denote the right-hand-side of~\eqref{eq:qE}.
We clearly have $q_A\le q(E)$.
We may assume that $\Bc$ is represented as a von Neumann subalgebra of bounded operators on a Hilbert space $\HEu$.
Suppose $\xi\in q(E)\HEu$.
Then for every $\eps>0$, there is $n\in\Nats$ and there are $b_1,\ldots,b_n\in\ker E$ and $\zeta_1,\ldots,\zeta_n\in\HEu$ such that
\[
\left\|\xi-\sum_{j=1}^nE(b_j^*b_j)\zeta_j\right\|_2<\eps.
\]
By applying Kaplansky's density theorem to the real and imaginary parts of $b_j$ individually,
we find that, for each $j$, there is a bounded sequence $(a_{j,k})_{k=1}^\infty$ in $A$ that converges in strong operator topology
to $b_j$ and such that $a_{j,k}^*$ converges in strong operator topology to $b_j^*$.
Thus, $E(a_{j,k}^*a_{j,k})$ converges in strong operator topology to $E(b_j^*b_j)$.
Since $E$ is continuous with respect to strong operator topology and $E(b_j)=0$, we may without loss of generality assume $E(a_{j,k})=0$
for each $j$ and $k$.
This implies that $\xi$ has distance no more than $\eps$ to an element of $q_A\HEu$.
This shows $q(E)\le q_A$, as required.
\end{proof}

\begin{thm}\label{thm:center}
As described above,
let $E:\Bc\to\Dc$ be a faithful normal conditional expectation of von Neumann algebras.
Let
\begin{equation}\label{eq:MF1}
(\Mcal,F)=(*_\Dc)_1^\infty(\Bc,E)
\end{equation}
be the von Neumann algebra free product with amalgamation over $\Dc$ of infinitely many copies of $(\Bc,E)$.
Then the center of $\Mcal$ is contained in $\Dc$ and is, in particular $Z(\Bc)\cap\Dc$.
\end{thm}
\begin{proof}
Let $x\in Z(\Mcal)$ and let us show $F(x)=x$.
We may without loss of generality assume that $\Bc$ is countably generated.
Indeed, the $*$-subalgebra of $\Mcal$ spanned by the words in the various copies of $\Bc$ is strong-operator-topology
dense in $\Mcal$, so $x$ is the limit (in strong-operator-topology) of a sequence, each term of which is a finite sum of such words.
Thus, we may replace $\Bc$ by some countably generated subalgebra of it, and it will suffice to show $F(x)=x$ under these hypothesis.
Note that, then, $\Dc$ is also countably generated.

It was proved by E.\ Ricard that $F$ is faithful.
(See Thm. 2.1~\cite{I11}, where this proof appears in print.)
Since $\Dc$ is countably generated, it has a faithful, normal state $\phi$.
Then $\phi\circ F$ is a normal, faithful state of $\Mcal$.
We will let $\Bc_i$ denote the $i$-th copy of $\Bc$ in $\Mcal$, arising from the 
amalgamated free product construction and $E_i:\Bc_i\to\Dc$ the corresponding copy of the conditional expectation $E$.
We will let $\HEu_i$ denote the Hilbert $\Dc$-bimodule $L^2(\Bc_i,E_i)$, which contains an orthocomplemented copy of $\Dc$,
and we will let $\HEu_i\oup$ denote the orthocomplement of $\Dc$ in $\HEu_i$.
Recall that $\Mcal$ is constructed as follows:
on the Hilbert $\Dc$-bimodule
\[
\HEu=\Dc\oplus\bigoplus_{\substack{k\ge1 \\ i_1,\ldots,i_k\ge1 \\ i_j\ne i_{j+1}}}
\HEu_{i_1}\oup\otimes_\Dc\cdots\otimes_\Dc\HEu_{i_k}\oup,
\]
we have the usual left actions of $\Bc_i$, $i\in \Nats$. 
Letting $\pi_\Dc$ denote the GNS representation of $\Dc$ on $L^2(\Dc,\phi)$, 
the left actions of $\Bc_i$ on $\HEu$ induce left actions on the Hilbert space
\[
\HEu\otimes_{\pi_\Dc}L^2(\Dc,\phi)
=L^2(\Dc,\phi)\oplus\bigoplus_{\substack{k\ge1 \\ i_1,\ldots,i_k\ge1 \\ i_j\ne i_{j+1}}}
\HEu_{i_1}\oup\otimes_\Dc\cdots\otimes_\Dc\HEu_{i_k}\oup\otimes_{\pi_D}L^2(\Dc,\phi),
\]
which is in a natural way identified with $L^2(\Mcal,\phi\circ F)$.

We will review some elements of Tomita--Takesaki theory that we need; see, for example~\cite{T03}.
For $x\in\Mcal$, we will write $\xh$ for the corresponding element of $L^2(\Mcal,\phi\circ F)$.
We let $\lambda$ denote the usual left action of $\Mcal$ on $L^2(\Mcal,\phi\circ F)$.
Thus, $\lambda(x)\yh=(xy)\hat{\:}$, $y\in M$.
Consider the modular automorphism group $\sigma=(\sigma_t)_{t\in\Reals}$ of $\Mcal$ associated to the faithful state $\phi\circ F$.
Note that $\sigma$ leave $\Dc$ and each $\Bc_i$ globally invariant.
Let $S_\phi$ denote the densely defined, closed, conjugate linear operator defined by $S_\phi(\xh)=(x^*)\hat{\;}$
and, as usual, write $S_\phi=J_\phi\Delta_\phi^{1/2}$ for its polar decomposition.
Then $J_\phi$ is a conjugate linear, surjective isometry.
For $x\in\Mcal$, let $r(x)$ denote the densely defined ``right multiplication'' operator $r(x)\yh=(yx)\hat{\;}$, $y\in M$.
It need not be bounded.
However, (see, for example equations (4)-(6) of~\cite{F00}), if $x\in \Mcal$ is entire analytic with respect to $\sigma$, then 
$r(x)=J_\phi\lambda(\sigma_{-i/2}(x^*))J_\phi$ and $r(x)$ is bounded.
Note, also, that if $x\in M$ is entire analytic, then so is $F(x)$.

Let $x\in Z(\Mcal)$ and let $\eta=\xh-F(x)\hat{\;}$.
We will show that $\eta=0$, and this will prove the theorem.
Let $\eta_N$ denote the orthogonal projection of $\eta$ onto 
the subspace
\[
\HEu^{(N)}=\bigoplus_{\substack{k\ge1 \\ 1\le i_1,\ldots,i_k\le N \\ i_j\ne i_{j+1}}}
\HEu_{i_1}\oup\otimes_\Dc\cdots\otimes_\Dc\HEu_{i_k}\oup\otimes_{\pi_D}L^2(\Dc,\phi),
\]
of $L^2(\Mcal,\phi\circ F)$.
Then $\lim_{N\to\infty}\|\eta-\eta_N\|=0$.
Suppose $b_1\in\Bc_1\cap\ker E_1$ is entire analytic with respect to the restriction of $\sigma$ to $\Bc_1$.
Let $b_N$ be the copy of $b_1$ in $\Bc_N$.
Since $x$ is central, we have
\begin{equation}\label{eq:bxxb}
0=(b_Nx-xb_N)\hat{\;}=(\lambda(b_N)-r(b_N))\xh=(\lambda(b_N)-r(b_N))(\eta_{N-1}+F(x)\hat{\;}+\eta-\eta_{N-1}).
\end{equation}
By considering the dense set of elements in $\HEu^{(N-1)}$ that are of the form $\yh$ for $y\in\Mcal$, we see that the three subspaces
\[
\lambda(b_N)\HEu^{(N-1)},\qquad r(b_N)\HEu^{(N-1)}\quad\text{and}\quad\lambda(b_N)L^2(\Dc,\phi)+r(b_N)L^2(\Dc,\phi)
\]
are mutually orthogonal.
Therefore, from~\eqref{eq:bxxb}, we get
\begin{multline}\label{eq:sumeta}
\|\lambda(b_N)\eta_{N-1}\|_2^2+\|r(b_N)\eta_{N-1}\|_2^2+\|(\lambda(b_N)-r(b_N))F(x)\hat{\;}\|_2^2 \\
=\|(\lambda(b_N)-r(b_N))(\eta-\eta_{N-1})\|_2^2\le\|\lambda(b_N)-r(b_N)\|^2\|\eta-\eta_{N-1}\|_2^2.
\end{multline}
For every permutation $\rho$ of $\Nats$, there is an $F$-preserving automorphism of $\Mcal$ that sends the copy $\Bc_i$
identically to the copy $\Bc_{\rho(i)}$.
Thus, we see that $\|\lambda(b_N)-r(b_N)\|$ does not depend on $N$.
So the right-hand-side of~\eqref{eq:sumeta} tends to $0$ as $N\to\infty$, and we get
$\lim_{N\to\infty}\|\lambda(b_N)\eta_{N-1}\|_2=0$.
Letting $d=E_1(b_1^*b_1)^{1/2}$, we have $d=E_N(b_N^*b_N)^{1/2}$ for every $N$.
Moreover, again by 
considering the dense set of elements in $\HEu^{(N-1)}$ that are of the form $\yh$ for $y\in\Mcal$, 
we have
$\|\lambda(b_N)v\|_2=\|\lambda(d)v\|_2$ for every $v\in\HEu^{(N-1)}$.
So we get
\[
\|\lambda(d)\eta\|_2=\lim_{N\to\infty}\|\lambda(d)\eta_{N-1}\|_2=\lim_{N\to\infty}\|\lambda(b_N)\eta_{N-1}\|_2=0.
\]
Since $E_N(b_N^*b_N)=d^2$ for all $N$, we have
\begin{equation}\label{eq:Ebeta0}
\lambda(E_N(b_N^*b_N))\eta=0.
\end{equation}

Let $A\subseteq\Bc$ be the $*$-subalgebra of elements that are entire analytic with respect to the
restriction to $\Bc_1$ of the modular automorphism group $\sigma$ (where we have identified $\Bc$ and $\Bc_1$).
Then $A$ is strong-operator-topology dense in $\Bc$
(see, for example, Prop. 2.5.22 of~\cite{BR87} or Lemma 2.3 of~\cite{T03}).
Let
\[
q=\bigvee\;\{[E(b^*b)]\mid b\in A\cap\ker E\}.
\]
By Lemma~\ref{lem:suppProj}, 
$q\in Z(\Bc)\cap\Dc$ and $(1-q)\Bc=(1-q)\Dc$.
All copies of $q$ in the various $\Bc_N$ are, therefore identified with each other, and $q$ lies in the center of $\Mcal$.
Thus, $(1-q)\Bc_N=(1-q)\Dc$ for all $N$, and we have $(1-q)\Mcal=(1-q)\Dc$.
From~\eqref{eq:Ebeta0}, we conclude that $q\eta=0$.
Since $\phi\circ F$ is faithful, this yields $q(x-F(x))=0$.
So $x-F(x)=(1-q)(x-F(x))\in(1-q)\Dc$.
However, $x-F(x)\perp\Dc$, so we have $x-F(x)=0$, as required.
\end{proof}

\section{The simplex of KMS quantum symmetric states}
\label{sec:simplex}

Turning back to the notation described in Sec.~\ref{sec:QSKMSS}, we let $K(*_1^\infty A,*_1^\infty\sigma)$
denote the set of all $*_1^\infty\sigma$-KMS states on the universal unital free product C$^*$-algebra $*_1^\infty A$.
This set $K(*_1^\infty A,*_1^\infty\sigma)$ is a compact affine set in the weak$^*$-topology and is known to be a Choquet simplex whose extreme points are those $\phi\in K(*_1^\infty A,*_1^\infty\sigma)$, the von Neumann algebras
generated by images of whose GNS-representations are factors
({\em cf}.\/ \cite{EKV70}, \cite{R84} and~\cite{TW73}).

Recall that we defined $\QSS_\sigma(A)=K(*_1^\infty A,*_1^\infty\sigma)\cap\QSS(A)$.

\begin{remark}
We easily see that $\QSS_\sigma(A)$ is nonempty if and only if $A$ possesses a $\sigma$-KMS state.
Indeed, if $\phi$ is a $\sigma$-KMS state on $A$, then we always have at least two elements of $\Vc_\sigma(A)$,
which are distinct so long as $A\ne\Cpx$:
\begin{enumerate}[(i)]
\item taking $(B,D,F,\theta,\gamma,\rho)=(A,\Cpx,\phi,\id_A,\sigma,\id_\Cpx)$, we get the free product state
\[
\Psi_\sigma(B,D,F,\theta,\gamma,\rho)=*_1^\infty\phi\in\QSS_\sigma(A);
\]
\item taking $(B,D,F,\theta,\gamma,\rho)=(A,A,\id_A,\id_A,\sigma,\phi)$, we get the state
\[
\Psi_\sigma(B,D,F,\theta,\gamma,\rho)=\phi\circ\big(*_1^\infty \id_A\big)\in\QSS_\sigma(A)
\]
that arises from the $*$-homomorphism $*_1^\infty A\to A$
sending each copy of $A$ identically to itself, followed by the state $\phi$.
\end{enumerate}
The paper~\cite{Wo85} addresses and answers the question of existence of a $\sigma$-KMS state of $A$.
\end{remark}

The following theorem follows easily, just like Thm.~5.3 of~\cite{DDM14},
since we have already proved the main ingredients.

\begin{thm}\label{thm:simplex}
If $\QSS_\sigma(A)$ is nonempty, then it is a Choquet simplex and a face of $K(*_1^\infty A,*_1^\infty\sigma)$.
Moreover, given $\psi\in\QSS_\sigma(A)$,
let $(B,D,F,\theta,\gamma,\rho)\in\Vc_\sigma(A)$ be the element that maps to $\psi$
under the bijection $\Psi_\sigma$ from Thm.~\ref{thm:Psisig}.
Let $\Bc$ be the von Neumann algebra generated by the image of $B$
under the GNS representation $\pi_{\rho\circ F}$ of the state $\rho\circ F$ and let
$\Dc$ be the von Neumann subalgebra generated by $\pi_{\rho\circ F}(D)$.
Then the following are equivalent:
\begin{enumerate}[(i)]
\item\label{it:extrQSS} $\psi$ is an extreme point of $\QSS_\sigma(A)$,
\item\label{it:extrKMS} $\psi$ is an extreme point of $K(*_1^\infty A,*_1^\infty\sigma)$,
\item\label{it:center} $Z(\Bc)\cap\Dc=\Cpx1$, where $Z(\Bc)$ is the center of $\Bc$.
\end{enumerate}
\end{thm}
\begin{proof}
The implication \eqref{it:extrQSS}$\implies$\eqref{it:extrKMS}, when proved, will imply that $\QSS_\sigma(A)$ is a face of 
$K(*_1^\infty A,*_1^\infty\sigma)$ and, thus, that it is a Choquet simplex.

The implication  \eqref{it:extrKMS}$\implies$\eqref{it:extrQSS} is clearly true.

We now show the implication \eqref{it:center}$\implies$\eqref{it:extrKMS}.
By construction, the state $\rho\circ F$ of $B$ is $\gamma$-KMS and faithful.
By Lemma~\ref{lem:faithful}, this implies that the normal extension $(\rho\circ F)\hat{\;}$ of $\rho\circ F$ to $\Bc$ is faithful.
Moreover, the conditional expectation $F$ induces a projection from $L^2(B,\rho\circ F)$ onto the subspace $L^2(D,\rho)$
and compression by this projection induces a $(\rho\circ F)\hat{\;}$-preserving, normal, conditional expectation $\Fc:\Bc\to\Dc$, that extends $F$. Letting $\rhohat$ denote the unique normal extension of $\rho$ to $\Dc$, we have $(\rho\circ F)\hat{\;}=\rhohat\circ\Fc$.
The image $\pi_\psi(*_1^\infty A)$ of the GNS representation of $\psi$ is, by the construction in Thm.~7.3 of~\cite{DKW},
isomorphic to the reduced amalgamated free product C$^*$-algebra
\[
(Y,G)=(*_D)_{i=1}^\infty(B,F),
\]
of infinitely many copies of $(B,F)$, with amalgamation over $D$.
Moreover, the GNS representation is itself the defining representation of the reduced free product C$^*$-algebra
on the Hilbert space $\HEu=L^2(Y,G)\otimes_{\pi_\rho}L^2(D,\rho)$,
where $L^2(Y,G)$ is the free product of infinitely many copies of the Hilbert $D$-bimodule $L^2(B,F)$.

The von Neumann algebra $\Mcal_\psi:=\pi_\psi(*_1^\infty A)''$ generated by the image of the GNS representation of $\psi$
is, therefore, isomorphic to the amalgamated free product of von Neumann algebras,
\[
(*_\Dc)_{i=1}^\infty(\Bc,\Fc),
\]
of infinitely many copies of $(\Bc,\Fc)$, with amalgamation over $\Dc$.
By Thm.~\ref{thm:center} and the hypothesis~\eqref{it:center}, it follows that $\Mcal_\psi$ is a factor.
Therefore, $\psi$ is an extreme point of $K(*_1^\infty A,*_1^\infty\sigma)$.

We will now show \eqref{it:extrQSS}$\implies$\eqref{it:center}, which will finish the proof.
Suppose that~\eqref{it:center} fails to hold.
Then there is a projection $p\in Z(\Bc)\cap\Dc$ that is neither $0$, nor $1$.
Thus, $t:=\rhohat(p)\in(0,1)$.
Note that $p$ lies in the center of $\Mcal_\psi$ (Thm.~\ref{thm:center}).
Consider the states $\rho_0$ and $\rho_1$ of $D$, defined, for $d\in D$, by
\[
\rho_0(d)=t^{-1}\rhohat(\pi_\rho(d)p),\qquad\rho_1(d)=(1-t)^{-1}\rhohat(\pi_\rho(d)p).
\]
We see that $\rho_0\ne\rho_1$, by considering a bounded sequence $(x_n)_{n=1}^\infty$ in $D$ such
that $\pi_\rho(x_n)$ converges in strong operator topology to $p\in\Dc$.
Consider, for each $i\in\{0,1\}$,  the state $\psi_i=\rho_i\circ\Fc\circ\pi_\psi$ of $*_1^\infty A$.
By Thm.~4.1 of~\cite{DDM14}, each is a quantum symmetric state of $A$.
Using that $p$ lies in the center of $\Mcal_\psi$, we have that each $\psi_i$ is also $(*_1^\infty\sigma)$-KMS.
Thus, $\psi_i\in\QSS_\sigma(A)$.
But we have $\psi=t\psi_0+(1-t)\psi_1$, and $\psi_0\ne\psi_1$.
Therefore,~\eqref{it:extrQSS} fails to hold.
This finishes the proof.
\end{proof}

\end{document}